\documentclass[a4paper,12pt]{amsart}

\usepackage{amsfonts, amssymb, amsthm, amsmath, bbm, tikz}
\usepackage[pdftex,colorlinks,hypertexnames=false,linkcolor=blue,urlcolor=blue,citecolor=blue]{hyperref}

 \usepackage{geometry}
\newtheorem{thm}{Theorem}[section]
\newtheorem*{thmm}{Main Theorem}
\newtheorem{prop}[thm]{Proposition}
\newtheorem{lem}[thm]{Lemma}
\newtheorem{cor}[thm]{Corollary}

\theoremstyle{remark}
\newtheorem{rem}[thm]{Remark}
\newtheorem{defn}[thm]{Definition}
\newtheorem{ex}[thm]{Example}

\newcommand{\C}{\mathbb{C}}
\newcommand{\N}{\mathbb{N}}
\newcommand{\Z}{\mathbb{Z}}
\renewcommand{\k}{\mathbbm{k}}

\newcommand{\A}{\mathbb{A}}
\renewcommand{\Z}{\mathbb{Z}}

\newcommand{\g}{\mathfrak{g}}

\renewcommand{\l}{\mathfrak{l}}

\newcommand{\m}{\mathfrak{m}}

\renewcommand{\v}{\mathfrak{v}}

\newcommand{\R}{\mathcal{R}}
\newcommand{\B}{\mathcal{B}}
\renewcommand{\S}{\mathcal{S}}

\newcommand{\V}{\mathcal{V}}
\newcommand{\K}{\mathcal{K}}
\newcommand{\M}{{\mathcal{M}}}

\newcommand{\ad}{\operatorname{ad}}
\newcommand{\Ad}{\operatorname{Ad}}
\newcommand{\tr}{\operatorname{tr}}
\newcommand{\Lie}{\operatorname{Lie}}
\newcommand{\Ker}{\operatorname{Ker}}

\newcommand{\Hom}{\operatorname{Hom}}
\newcommand{\gr}{\operatorname{gr}}
\renewcommand{\th}{{\operatorname{th}}}
\newcommand{\Spec}{\operatorname{Spec}}
\newcommand{\res}{{\operatorname{res}}}

\newcommand{\Out}{\operatorname{Out}}
\newcommand{\Aut}{\operatorname{Aut}}
\newcommand{\Inn}{\operatorname{Inn}}

\newcommand{\Rees}{\operatorname{Rees}}
\newcommand{\Fun}{\operatorname{Fun}}
\newcommand{\Mat}{\operatorname{Mat}}
\newcommand{\Gr}{\operatorname{Gr}}
\newcommand{\Max}{\operatorname{Max}}

\newcommand{\om}{\overline{m}}

\newcommand{\oP}{{\overline{\Phi}}}

\newcommand{\orr}{{\overline{r}}}
\newcommand{\oPhi}{\overline{\Phi}}
\newcommand{\oOq}{\overline{\mathcal{O}}_q}
\newcommand{\oZo}{\overline{Z}_0}

\newcommand{\ua}{{\underline{a}}}
\newcommand{\ub}{{\underline{b}}}
\newcommand{\uc}{{\underline{c}}}

\newcommand{\uell}{{\underline{\ell - 1}}}

\newcommand{\tPhi}{\widetilde{\Phi}}
\newcommand{\tPsi}{\widetilde{\Psi}}

\newcommand{\oH}{{\Hom_\R(\M,\R)}}

\begin{document}

\title{Transfer results for Frobenius extensions}
\author{Stephane Launois \& Lewis Topley}
\address{}
\email{s.launois@kent.ac.uk}
\email{l.topley@kent.ac.}
\address{Sibson Building, The University of Kent, Canterbury, CT2 7NZ, UK}

\maketitle

\newcounter{parno}
\renewcommand{\theparno}{{\sf \thesection.\arabic{parno}}}
\newcommand{\p}{\smallskip \refstepcounter{parno}\noindent \theparno \space }
\newcommand{\clear}{\setcounter{parno}{0}}

\begin{abstract}
We study Frobenius extensions which are free-filtered by a totally ordered, finitely generated
abelian group, and their free-graded counterparts. First we show that the Frobenius property passes
up from a free-graded extension to a free-filtered extension, then also from a free-filtered
extension to the extension of their Rees algebras. Our main theorem states that,
under some natural hypotheses, a free-filtered extension of algebras is Frobenius if and only if
the associated graded extension is Frobenius.
In the final section we apply this theorem to provide new examples and non-examples
of Frobenius extensions.
\end{abstract}

\section{Introduction}
Throughout this paper $\k$ is a field of any characteristic. A finite dimensional algebra $\R$ over $\k$ is called a Frobenius
algebra if the dual of the right regular module is isomorphic to the left regular module $(\R_\R)^* \cong {}_\R\R$.
Equivalently $\R$ admits a linear map $\R\rightarrow \k$ whose kernel contains no left or right ideals - we call this the Frobenius form of $\R$.
It was first observed by Nakayama that the representation theory of Frobenius algebras admits extremely nice duality properties.
For instance, it is known that the projective and injective modules coincide and, in particular,
the left regular module is injective. Three notable examples include the group algebras of finite groups, reduced enveloping algebras
of restricted Lie algebras and semidirect products $\R \ltimes \R^*$ where $\R$ is any Artinian ring \cite[pp. 127]{ARS}, \cite[Proposition~1.2]{FP88}. 

Some of the most interesting algebras arising in representation theory are free modules of finite type over a central affine subalgebra $\S \subseteq \R$.
Under very mild assumptions $\S$ will act via $\k$ on the simple modules and so one is led to consider the family of algebras $\{\R/\m \R \mid \m \in \Max \S\}$,
 and when these quotients are Frobenius algebras the representation theory of $\R$ is somewhat simplified. We say that a $\k$-algebra $\R$ is a free Frobenius
extension of a subalgebra $\S$ when $\R$ is a free left $\S$-module and there exists a form $\Phi :\R \rightarrow \S$ generalising the Frobenius form of a Frobenius algebra (see \textsection\ref{thefrobform}).
When $\S\subseteq \R$ is a central extension which is Frobenius, the family of quotients $\R / \m \R$ over the maximal spectrum $\Max \S$ are all Frobenius
(apply Lemma~\ref{quotientslemma} for example). Thus we view Frobenius extensions as a natural and useful generalisation of Frobenius algebras.

Brown--Gordon--Stroppel gave many new examples of Frobenius extensions \cite{BGS06}. Their approach was fairly uniform:
in each case they gave an example of a form $\Phi : \R \rightarrow \S$ and checked the Frobenius property via a single simple hypothesis.
In \cite[1.6]{BGS06} they asked whether there exists an axiomatic approach which would apply to all of their examples simultaneously,
and it was this question which provided the first motivation for our work. One feature shared by many of their examples,
as well as other classical examples, is a filtration by a totally ordered finitely generated abelian group $G$,
and in this paper we develop general tools which might help to prove the Frobenius property in the presence of such a filtration.

Let $G$ be as above and suppose that $\R$ is a $G$-graded $\k$-algebra and that $\S\subseteq \R$ is a
graded subalgebra such that $\R$ is a free-graded left $\S$-module. We say that $\S \subseteq \R$ is a free-graded Frobenius extension if $\R$ comes
equipped with a homogeneous Frobenius form $\Phi: \R \rightarrow \S$. Similarly we say that $\S \subseteq \R$ is a free-filtered Frobenius extension
if $\R$ is a $G$-filtered algebra, a free-filtered $\S$-module and $\S \subseteq \R$ is a Frobenius extension. When $\R$ is a $G$-filtered algebra we
write $\gr \R$ for the associated graded algebra.

The following is the main theorem of this paper:
\begin{thmm}
Suppose that $\S \subseteq \R$ is a free-filtered extension. Then $\S \subseteq \R$ is a free-filtered Frobenius extension if and only if $\gr \S\subseteq \gr \R$ is a free-graded
Frobenius extension.
\end{thmm}

Our method is to show that the Frobenius property passes up from a free-graded extension to a free-filtered extension (any choice of filtered
lift of the homogeneous Frobenius form will suffice) and then show that the Rees algebra of a free-filtered Frobenius extension is naturally a free-graded Frobenius extension.
For the latter part we define the Rees algebra $\Rees(\R)$ for an algebra $\R$ filtered by a totally ordered, finitely generated abelian group,
and observe that $\Rees(\R)$ simultaneously deforms $\R$ and $\gr \R$ (Lemma~\ref{reducingRees}), a useful property which generalises the well-known situation where $G = \Z$.

In Section~\ref{freefrobextensions} we provide the general background on free-graded (free-filtered) modules and algebras. In Section~\ref{frobsec}
we give the precise definition of free-graded (free-filtered) Frobenius extensions and define certain invariants of such extensions: the rank, the degree
and the Nakayama automorphism. We also give a brief account of the Rees algebra and prove the deformation property mentioned previously. 
In Section~\ref{proofof} we state and prove various transfer results and combine them to deduce a proof of the main theorem.
Finally, in Section~\ref{appsec} we give a few applications of the main theorem, describing new examples and non-examples of Frobenius extensions.
For our first example we consider a quantum Schubert cells at
an $\ell$th root of unity, and we show that these are all Frobenius extensions of their $\ell$-centre. Next we consider modular finite $W$-algebras,
first defined by Premet in \cite{Pr02} and recently studied by Goodwin and the second author in \cite{GT17}. We show that
the modular $W$-algebra in characteristic $p > 0$ is a free-filtered Frobenius extension of its $p$-centre, with trivial Nakayama automorphism. In the final section
we use the transfer result to show that the quantum Grassmanian $\mathcal{G}r(2,4)$ at an $\ell$th root of unity actually fails to
be a free-Frobenius extension of its $\ell$-centre.
We note that although it is a reasonably tractable problem to show that a given extension of algebras is Frobenius, we do not know of any methods
in the current literature which allow you to prove that an extension is not Frobenius. Our key observation is that if $\S \subseteq \R$ is a
$G$-filtered Frobenius extension and $\R$ is a free-filtered $\S$-module then the multiset of filtered degrees of the $\S$-basis for $\R$
must exhibit a special symmetry (Lemma~\ref{symmetricdegrees}), an observation which we use repeatedly in different guises throughout the paper.

\medskip

\noindent {\bf Acknowledgements:} Both authors are grateful for funding received from EPSRC grant EP/N034449/1. Furthermore the second author would like to to acknowledge
the support of the European Commission, Seventh Framework Programme, Grant Agreement 600376, as well as grants CPDA125818/12, 60A01-4222/15 and DOR1691049/16
from the University of Padova.

\section{Preliminaries on gradings and filtrations}\label{freefrobextensions}

The purpose of this paper is to study Frobenius extensions $\S \subseteq \R$ when $\R$ is a finitely generated filtered $\k$-algebra
and $\S$ is a subalgebra. In order to make our results as broadly applicable as possible we work with filtrations over ordered groups.
For the reader's convenience we have provided a brief introduction to  the theory of algebras and modules endowed with such filtrations.
When we say ``module'' we mean ``left module'' unless otherwise stated.

All algebras are assumed to have the invariant basis number property, ie. $\R^n \cong \R^m$ as left modules implies $n = m$.
By considering the quotient $\R^n/\m \R^n$ where $\m$ is a maximal ideal of $\R$ it is clear that all commutative rings possess this property.
Similarly, every ring which is a finite free module over a commutative subring possesses this property, and such rings constitute the majority of our examples.

\subsection{Totally ordered groups}\label{totordgp}
Throughout this paper $(G, \leq)$ shall always be a totally ordered, finitely generated abelian group. Note that such a group
is obviously torsion free, hence free. All applications which we have in mind involve
finitely many copies of $\Z$ ordered lexicographically. Much later we shall need to consider the group algebra $\k G$, and so we use multiplicative notation for $G$,
writing $1_G$ for the identity element.

\subsection{Filtrations by groups}\label{filtbygp}
A $G$-filtration of a $\k$-vector space $\V$ is a collection of subspaces $(\V_g : g\in G)$ satisfying $\V_g \subseteq \V_h$ whenever $g \leq h$.
Throughout this article we refer to $G$-filtrations simply as filtrations however, since $G$ shall be fixed throughout, this shall cause no confusion.
We say that the filtration is:
\begin{itemize}
\item{\emph{$\k$-finite} if $\dim \V_g < \infty$ for all $g\in G$;} 
\item{\emph{discrete} if $\V_g = 0$ for some $g$;}
\item{\emph{non-negative} if $\V_g = 0$ for $g < 1_G$;}
\item{\emph{exhaustive} if $\bigcup_{g\in G} \V_g = \V$;}
\item{\emph{proper} if $\V_g \neq \V$ for all $g$.}
\end{itemize}
Throughout this paper all filtrations are assumed to be $\k$-finite, discrete, exhaustive and proper.

For $0 \neq v \in \V$ the hypothesis that the filtration is discrete implies that $\{g\in G: v \in \V_g\}$ has a minimum in $G$, which we define to
be the degree $\deg(v)$. The degree of $0 \in \V$ is taken to be $-\infty$ which, by convention, shall satisfy $-\infty < g$ for all $g\in G$.
If $\R$ is a $\k$-algebra then a filtration of $\R$ should satisfy the additional hypothesis that $\R_g \R_h \subseteq \R_{gh}$ for all $g,h\in G$.

\begin{ex}\label{standardfilt}
Let $F :=\k\langle X_i \mid i \in I \rangle$ be the free associative algebra on generators $\{ X_i \mid i \in I\}$, which we think of as polynomials in non-commuting variables.
For each tuple $(d_i \mid i\in I)$ of elements of $G$, we define a grading $F := \bigoplus_{g\in G } F^g$ by placing $X_i$ in degree $d_i$. This induces a filtration by 
$F_g := \bigoplus_{h \leq g} F^h$. Now suppose that $\R$ is generated by $\{r_i\in \R \mid i\in I\}$
and define a homomorphism $\eta : F \rightarrow \R$ by $X_i \mapsto r_i$ for all $i$. This endows $\R$ with a filtration by setting $$\R_g := \eta(F_g).$$
Filtrations defined in this manner are called \emph{standard filtrations}.
\end{ex}

\subsection{Gradings by groups}

A $G$-grading (or just a grading) of a vector space is a decomposition $V = \bigoplus_{g\in G} V^g$. 
All gradings in this article are assumed to be $\k$-finite, exhaustive, non-negative and proper, and these
conditions are defined analogously to those same conditions on filtrations. Every grading $V = \bigoplus_{g\in G} V^g$
induces a filtration $V_g := \bigoplus_{h \leq g} V^h$. If we have a filtered vector space $\V$ then
we may define the associated graded space $V = \gr \V$ by setting $V^g = \V_g / (\sum_{h< g} \V_h)$
and $V = \bigoplus_{g\in G} V^g$. We define the degree function 
$\deg : \bigcup_{g\in G} V^g \rightarrow G\cup \{-\infty\}$ of homogeneous element by setting $\deg V^g\setminus \{0\} = g$ and $\deg(0) = -\infty$.

If $R$ is a graded $\k$-algebra then we insist that gradings satisfy $R^g R^h \subseteq R^{gh}$ for all $g,h \in G$. When $\R$ is a filtered algebra
$\gr \R$ is a graded algebra in the obvious manner. If $M$ is a graded $R$-module and $d\in G$ then we can define the shifted module $M[d]$
by setting $M[d] \cong M$ as $R$-modules and $M[d]_g:= M_{dg}$ as graded spaces.

\subsection{Free-filtered and free-graded modules}

If $\R$ is a filtered $\k$-algebra then a free-filtered $\R$-module is a filtered module $\M$ which is free on some basis $\{m_i \mid i \in I\}$
for which there exist elements $d_i \in G$ for every $i \in I$ such that
\begin{eqnarray}\label{filteredgens}
\M_g = \bigoplus_{i \in I} \R_{gd_i^{-1}} m_i.
\end{eqnarray}
We call a basis of a free-filtered module satisfying (\ref{filteredgens}) a \emph{free-filtered basis}.
\begin{rem}
It is worth noting
that not every basis for a free-filtered $\R$-module is a free-filtered basis. For example, if $\M = \R m_1 \oplus \R m_2$
is free-filtered and generated in degrees $d_1, d_2$ with $d_1 > d_2$ then $\{m_1, m_1 + m_2\}$ is not a free-filtered basis.
\end{rem}

If $R$ is a graded $\k$-algebra then a free-graded $R$-module is a graded module $M$ admitting
a homogeneous basis $\{m_i \in M \mid i \in I\}$. In other words free-graded modules are precisely those of the form
$\bigoplus_{i \in I} \R[d_i]$ for tuples $(d_i \mid i \in I) \in G^I$. As a consequence,
\begin{eqnarray}\label{gradeddims}
M^g = \bigoplus_{i \in I} R^{gd_i^{-1}} m_i
\end{eqnarray}
for all $g\in G$.

\begin{lem}\label{fgvsff}
Let $\M$ be a finitely generated filtered $\R$-module and $M = \gr \M$, $R = \gr \R$. Then $\M$ is free-filtered over $\R$ if and only if $M$ is free-graded over $R$.
In this case they have the same rank.
\end{lem}
\begin{proof}
If $\M = \bigoplus_{i \in I} \M_i$ is any direct sum decomposition then $M \cong \bigoplus_{i \in I} \gr \M_i$, and so it suffices to work in the rank one case
when proving the `only if' part. If $\M_g = \R_g m$ for some $m \in \M$ then $$\gr \M = \bigoplus_{g\in G} (\R_g m / \sum_{h < g}\R_{h} m) \cong (\gr \R)[\deg m]$$
as graded $\gr \R$-modules. Conversely, if $\{\om_i \in M \mid i \in I\}$ are homogeneous generators of $M$ then we can choose lifts $\{m_i \in \M \mid i\in I\}$
satisfying $m_i + \sum_{g < d_i} \M_{g} = \om_i$ for all $i$, where $d_i = \deg \om_i$. The argument of \cite[Lemma~4.7(2)]{To16} shows that
$\M$ is a free $\R$-module, whilst an easy induction using (\ref{gradeddims}) shows that $\M_g = \bigoplus_{i \in I} \R_{gd_i^{-1}} m_i$.
Note that this latter step relies on the filtration being discrete and $\k$-finite.
\end{proof}
\begin{lem}\label{multisetswelldefined}
The following hold:
\begin{enumerate}
\item[(i)] The multiset of degrees of any free-filtered basis of a finitely generated free-filtered $\R$-module $\M$ are uniquely determined by $\M$;
\item[(ii)] The multiset of degrees of any homogeneous basis of a finitely generated free-graded $\R$-module $M$ are uniquely determined by $M$;
\end{enumerate}
\end{lem}
\begin{proof}
Let $F:= \Fun(G, \Z)$ denote the abelian group of all functions $G \rightarrow \Z$. The group $G$ acts on $F$ by left translations, and we claim that if $f \neq 0$, and if $f(g) = 0$ for all $g < g' \in G$ 
then $G f$ generates a free abelian subgroup of $F$. The proof is by induction on the length of any $\Z$-dependence.
Suppose $\sum_{i\in I} a_i (g_i.f) = 0$ for integers $a_i \in \Z$ and some finite set $I$. If $|I| = 1$ then $f \neq 0$ settles the claim.
If $|I| > 1$ then choose $i \in I$ so that $g_{i}$ is maximal in the ordering. Then by assumption $0 = (\sum_{j\in I} a_j (g_j.f))(g_i g') = \sum_{j\in I} a_j f(g_j^{-1}g_i g') = a_i f(g')$
since $f(g_j^{-1}g_i g') = 0$ for $j < i$ by assumption. From $f(g') \neq 0$ we deduce $a_i = 0$ and so we have a dependence over the set $I \setminus \{i\}$.
Now apply the inductive hypothesis.

Suppose $\{c_i \mid i \in I'\}$ and $\{d_i \mid i \in I\}$ are both multisets of degrees of the free-filtered module $\M$.
For $i \in I$ write $f_{c_i}(g) := \dim \R_{c_i^{-1} g}$ and similar for $f_{d_i}$.
Since we have assumed that $\R$ has invariant basis number we may suppose $I = I'$. By equation (\ref{filteredgens}) we know that the function 
\begin{eqnarray*}
f & : & G \rightarrow \Z_{\geq 0}\\
& & g \longmapsto \dim \M_g
\end{eqnarray*}
is equal to $\sum_{i \in I} f_{c_i} = \sum_{i \in I} f_{d_i}$ and so $\sum_{i \in I} (f_{c_i} - f_{d_i}) = 0$ is a linear dependence between the $G$-translates
of $f \in F$. By the first paragraph the dependence is trivial, which proves (i). Part (ii) is proven similarly using equation (\ref{gradeddims}) instead of (\ref{filteredgens}).  
\end{proof}

Before we proceed we shall need a technical lemma.
\begin{lem}\label{technical}
Let $\R$ be a filtered $\k$-algebra and let $\M_1, \M_2$ be finite, free-filtered $\R$-modules.
Let $D_1$ and $D_2$ denote the multisets of degrees of any free-filtered basis for $\M_1$ and $\M_2$. 
Suppose also that $\Psi : \M_1 \rightarrow \M_2$ is injective (resp. surjective) of degree $d$. Then it is bijective if and only if $D_1d = D_2$.
\end{lem}
\begin{proof}
Under these hypotheses $\Psi$ is bijective if and only if $\dim (\M_1)_g = \dim (\M_2)_{gd}$ for all $g\in G$. It follows from
equation~(\ref{filteredgens}) that for $j =1,2$, $\dim (\M_j)_g = \sum_{i\in I} \dim \R_{g d_i^{-1}}$ which is completely determined by the multiset
$(d_i \mid i\in I)$ of degrees. This last claim requires the filtration to be discrete. 
\end{proof}

\subsection{$\Hom$ spaces between free-filtered modules}\label{filtHom}

Here we assume that $\R$ is a non-negatively filtered $\k$-algebra and that $\M$ is a finite, free-filtered module with basis $m_1,...,m_n$
in filtered degree $d_1,...,d_n \in G$ respectively. We are interested in the space $\Hom_\R(\M, \R)$

There is a natural filtration on $\Hom_\R(\M, \R)$ defined by
$$\Hom_\R(\M,\R)_g := \{ \phi : \M\rightarrow \R \mid \phi(\M_{h}) \subseteq \R_{g h} \text{ for all } h \in G \}.$$

We define special elements $\phi_1,...,\phi_n \in \Hom_\R(\M, \R)$ by $$\phi_i(m_j) = \delta_{i,j}$$ and extending by $\R$-linearity.
The next result is straightforward.
\begin{lem}\label{homfilt}
We have $\Hom_\R(\R m_i, \R) = \R \phi_i$ for each $i$, and $$\Hom_\R(\M, \R) = \bigoplus_{i=1}^n \Hom_\R(\R m_i, \R)$$ is a free-filtered
$\R$-module with free-filtered basis $\phi_1,...,\phi_n$ lying in degrees $d_1^{-1},...,d_n^{-1}$.$\hfill \qed$
\end{lem}

\subsection{The associated graded of the $\Hom$-space}

We continue to assume the hypotheses of the previous subsection, setting $R = \gr \R$ and $M = \gr \M$.
Observe that both $M$ and $\Hom_R(M, R)$ are $R$-modules.
According to Lemma~\ref{homfilt} the space $\Hom_\R(\M,\R)$
is a filtered $\R$-module, and so we may define the graded $R$-module $\gr \oH$.
\begin{lem}\label{identifygraded}
\begin{enumerate}
\item{The space $\Hom_{R}(M, R)$ is naturally graded;}
\item{There exists a natural isomorphism of graded $R$-modules $$\gr \Hom_\R(\M,\R) \overset{\sim}{\longrightarrow} \Hom_{R}(M, R).$$}
\end{enumerate}
\end{lem}
\begin{proof}
By Lemma~\ref{homfilt} and Lemma~\ref{fgvsff}, $\gr \Hom_\R(\M,\R)$ is free-graded and using the graded version of Lemma~\ref{homfilt} we see that $\Hom_{R}(M, R)$ is free-graded.
This also shows that the degrees of the free-graded generators are equal to $d_1^{-1},...,d_n^{-1}$ where $d_1,...,d_n$ is the sequence of degrees of $\M$ as a free-filtered $\R$-module.
The isomorphism class (as a graded module) of a free-graded module is entirely determined by the multiset of degrees of the generators, hence the result.
\end{proof}

\section{Frobenius extensions, Nakayama automorphisms and Rees algebras}\label{frobsec}

\subsection{Free Frobenius extensions}\label{freefrobenius}
Let $\S \subseteq \R$ be $\k$-algebras. The space $\Hom_\S(\R, \S)$ is a $\R$-$\S$-bimodule by the rules
\begin{eqnarray*}
(f.s)(r) := f(r)s;\\
(q.f)(r) := f(r q)
\end{eqnarray*}
where $q,r \in \R$, $s \in \S$ and $f \in \Hom_\S(\R, \S)$. Similarly, $\Hom_\S(\R,\S)$ is an $\S$-$\R$-bimodule.
We say that the inclusion $\S \subseteq \R$ is a \emph{Frobenius extension} if
\begin{enumerate}
\item[(i)]{$\R$ is a projective (left) $\S$-module;}
\item[(ii)]{the $\R$-$\S$-bimodules $\R$ and $\Hom_\S(\R,\S)$ are isomorphic.}
\end{enumerate}
Cases where $\R$ is a finite free $\S$-module are prolific in modern algebra (see \cite[III.4]{BG02} for an overview) and these so-called
\emph{free Frobenius extensions} shall be our main object of study. Since we always assume that our rings have invariant basis number
we may define \emph{the rank of a free Frobenius extension} to be the rank of $\R$ as an $\S$-module.

These extensions generalise the concept of Frobenius algebras over $\k$, which are precisely the free Frobenius extensions where $\S = \k$.
The theory was originally motivated by the extremely nice duality properties exhibited in the representation theory of Frobenius algebras.
For example, the injective and projective modules coincide for such algebras \cite[Theorem~62.3]{CR66}.

\subsection{The Frobenius form}\label{thefrobform}
It is well known that Frobenius algebras over $\k$ are characterised by the existence of a 1-form $\R \rightarrow \k$ such that the kernel contains no
non-zero proper ideals. We shall work with a similar characterisation of free Frobenius extensions which was observed by Nakayama and Tsuzuku.
\begin{lem}\label{characterisation} \cite{NT59}
A finite, free extension $\S\subseteq \R$ is Frobenius if and only if there exists a map $\Phi : \R\rightarrow \S$ of left $S$-modules such that
\begin{enumerate}
\item[(F1)]{the kernel of $\Phi$ contains no proper (left or right) ideals;}
\item[(F2)]{the assignment
\begin{eqnarray*}
\R \rightarrow \Hom_\S(\R, \S)\\
r \mapsto (q \mapsto \Phi(rq))
\end{eqnarray*}
is surjective.}
\end{enumerate}
\end{lem}
The form $\Phi$ shall be called the \emph{Frobenius form} of the extension. The proof of the lemma is easy to describe: supposing $\R$ 
is isomorphic to $\Hom_\S(\R,\S)$ as an $\R$-$\S$-bimodule, the unit $1\in \R$ induces a special element of $\Hom_\S(\R,\S)$ satisfying (1) and (2).
Conversely, the assignment described in (2) is certainly injective whenever (1) holds. For more detail see \cite[pp.11]{NT59}.

\begin{ex}
Some classical examples of Frobenius extensions are:
\begin{itemize}
\item{the enveloping algebra of a restricted Lie algebra in characteristic $p > 0$, viewed as an extension of the $p$-centre \cite[Proposition~1.2]{FP88};}
\item{the quantised enveloping algebra at an $\ell$th root of unity, viewed as an extension of its $\ell$-centre \cite{Ku94};}
\item{more generally, for every PI Hopf triple $Z_0 \subseteq Z \subseteq R$ the extension $Z_0 \subseteq R$ is Frobenius \cite[Corollary~III.4.7]{BG02}.
Notice that this example subsumes the previous two.}
\item{Numerous examples were discussed by Brown--Gordon--Stroppel in \cite{BGS06}. Finding machinery which could be applied to large
families of examples simultaneously was a motivating goal of the current article. }
\end{itemize}
\end{ex}

\subsection{Free-filtered and free-graded Frobenius extensions}\label{freegrad}

We continue with a totally ordered, finitely generated abelian group $G$. We remind the reader that all gradings and filtrations in this article are
$\k$-finite, proper, discrete and exhaustive. Let $\R$ be a non-negatively filtered $\k$-algebra and $\S \subseteq \R$ a subalgebra. 
\begin{defn}
We call $\S \subseteq \R$ a \emph{free-filtered Frobenius extension}
if $\R$ is a free-filtered $\S$-module and $\S \subseteq \R$ is a Frobenius extension.
\end{defn}

Now let $\R$ be a graded $\k$-algebra and suppose the subalgebra $\S \subseteq \R$ inherits the grading.
When $\R$ is a finitely generated $\S$-module the space $\Hom_\S(\R, \S)$ is graded (\cite[Theorem~1.2.6]{Ha15}).
\begin{defn}
We say that $\S \subseteq \R$ is a \emph{free-graded Frobenius extension}
if $\S \subseteq \R$ is a free Frobenius extension with Frobenius form $\Phi : \R \rightarrow \S$ such that 
\begin{enumerate}
\item[(GF1)] {$\Phi$ is homogeneous;}
\item[(GF2)] {$\R$ is a free-graded (left) $\S$-module.}
\end{enumerate}
In this situation, $\Phi$ shall be called the \emph{homogeneous Frobenius form} of the extension.
\end{defn}

\begin{rem}\label{altdefremark}
\begin{itemize}
\item[(i)] When a graded algebra $\R$ is a finite free-graded module over the subalgebra $\S$ and $\Phi : \R \rightarrow \S$ is a homogeneous
homomorphism of left $\S$-modules, axiom (F1) is equivalent to 
\begin{enumerate}
\smallskip
\item[(F1$^\prime$)]{$\Ker(\Phi)$ contains no proper graded ideals.}
\smallskip
\end{enumerate}
\item[(ii)] One remarkable corollary of our transfer result is that an extension $\S \subseteq \R$ is a free-graded Frobenius extension if and only if it a Frobenius extension and $\R$ is a free-graded $\S$-module (Corollary~\ref{simplifydefn}).
\end{itemize}
\end{rem}
\subsection{Invariants of filtered and graded Frobenius extensions}

The \emph{rank of a Frobenius extension} $\S \subseteq \R$ is defined to be the rank of $\R$ as an $\S$-module.
If $\S \subseteq \R$ is actually a free-filtered (resp. free-graded) extension then we define the \emph{degree of the Frobenius extension}
to be filtered (resp. graded) degree of any choice of (homogeneous) Frobenius form $\Phi : \R \rightarrow \S$.
\begin{lem}
The degree of a free-filtered or free-graded Frobenius extension is a well-defined invariant.
\end{lem}
\begin{proof}
The graded claim may be seen as a special case of the filtered claim, and so we prove the latter.
Observe that if $\Phi : \R \rightarrow \S$ is a homomorphism of left $\S$-modules of filtered degree $d$ then the map
$\Psi: \R \rightarrow \Hom_\S(\R,\S)$ given by $r \mapsto (q \mapsto \Phi(rq))$ is also of degree $d$.
Note that both $\R$ and $\Hom_\S(\R,\S)$ are free-filtered $\S$-modules and write $D_1$, $D_2$ for their
multisets of filtered degrees (these are well-defined thanks to Lemma~\ref{multisetswelldefined}). In Lemma~\ref{technical} we observed that $D_1 d = D_2$. Since $D_1$ and $D_2$
are bounded this equation holds for at most one $d \in G$, hence $d = \deg(\Phi)$ is determined by $\S\subseteq \R$.
\end{proof}

\subsection{Example: quantum affine space}\label{qas}
Since graded Frobenius extensions have not been studied explicitly before, we shall provide one detailed example.

Fix $\ell, n \in \N$ and an $n\times n$ matrix $q = (q_{i,j})_{1\leq i,j\leq n}$ with entries in $\k$ satisfying
$q_{i,j} = q_{j,i}^{-1}$ whenever $i\neq j$ and $q_{i,i} = 1$. We also suppose $q_{i,j}^\ell = 1$ for all $i,j$.
Then $A = \k_q[\A^n] := \k_q[x_1,...,x_n]$ shall denote the $n$-dimensional quantum affine space with generators
$x_1,...,x_n$ satisfying $x_i x_j = q_{i,j} x_{j} x_i$. Let $G$ be any abelian group
and choose elements $d_1,...,d_n \in G$. We view $A$ as a $G$-graded algebra by declaring that each $x_i$ lies in degree $d_i$.
The degree of any homogeneous element $a\in A$ is written $\deg(a)$.

The $\ell$-centre of $A$ is the graded (unital) subalgebra generated by elements $x_1^\ell,...,x_n^\ell$, and is
denoted $Z_0$. For $\ua = (a_1,...,a_n) \in \Z_{\geq 0}^n$, the monomial $x_1^{a_1} \cdots x_n^{a_n}$ in $A$
shall be denoted $x^\ua$. Consider the set of \emph{restricted monomials} $$\B^\res := \{x^\ua \, \mid\,  0 \leq a_i < \ell \text{ for all } i\}$$
and notice that $A$ is a free $Z_0$-module with basis $\B^\res$. The index $\uell = (\ell-1,....,\ell-1)$ plays a special role as follows.
We have a graded $Z_0$-module decomposition $A = \bigoplus_{\ua\in \B^\res} Z_0 x^\ua$ and we let $\Phi$ denote the projection onto
the factor corresponding to $x^\uell$, which is a homogeneous homomorphism of $Z_0$-modules.
\begin{prop}\label{freegradedqas}The following hold:
\begin{enumerate}
\item{$A$ is a free-graded Frobenius extension of $Z_0$;}
\item{$\Phi$ is a homogeneous Frobenius form;}
\item{the degree of the extension is $-\sum (\ell-1) d_i$;}
\item{the rank of the extension is $\ell^{n}$.}
\end{enumerate}
\end{prop}
\begin{proof}
It is easy to see that $A$ is a free-graded $Z_0$-module of rank $\ell^n$ with basis $\B^\res$, and that $\Phi$ is a 
homogeneous $Z_0$-equivariant map of the requisite degree, so it will suffice to confirm axioms (F1) and (F2)
for $\Phi$.

Suppose that
$$\Ker(\Phi) = \bigoplus_{x^\ua \in \B^\res ; \ua \neq \uell } Z_0 x^\ua$$
contains a right ideal $J$. Let $0\neq y\in J$. We can write $y= \sum z_\ua x^\ua$ for coefficients $z_\ua \in Z_0$,
and suppose $z_{\ub} \neq 0$ for some fixed $\ub = (b_1,...,b_n)$. Right multiplication by any monomial $x^\uc$
permutes the summands in the decomposition $A = \bigoplus Z_0 x^\ua$ and so if we let $\uc = (\ell-1-b_1,...,\ell-1-b_n)$
then $\Phi(y x^{\uc}) \neq 0$, contradicting the fact that $y x^{\uc} \in J \subseteq \Ker(\Phi)$. Similarly,
$\Ker(\Phi)$ contains no left ideals, using the decomposition of $A$ as a free right module over $Z_0$.

In order to see that $y \mapsto (z \mapsto \Phi(yz))$ is surjective $A \twoheadrightarrow \Hom_{Z_0}(A, Z_0)$ we show that
for each $\ub$ the projection $\sum z_\ua x^\ua \mapsto z_{\ub}$ lies in the image. The result will follow since $\Hom_{Z_0}(A, Z_0)$
is a free $Z_0$-module generated by these projections. Take $\uc := \uell - \ub$ as above
and observe that $y \mapsto \Phi(y x^{\uc})$ is a non-zero scalar multiple of the projection $\sum z_\ua x^\ua \mapsto z_{\ub}$, which completes the proof.
\end{proof}

\subsection{The Nakayama automorphism}
Let $\S \subseteq \R$ be a Frobenius extension with $\S$ central and Frobenius form $\Phi : \R \rightarrow \S$.
Thanks to \cite[Proposition~1]{NT60} we know that $\Hom_\S(\R, \S)$ is isomorphic to $\R$ as an $\R$-$\S$-bimodule and 
as an $\S$-$\R$-bimodule. In each case, $\Hom_\S(\R,\S)$ is generated by $\Phi$ as a (left or right) $\R$-module \cite[Section~1]{Ka61}.
This implies that there exists a bijection $\nu : \R \rightarrow\R$ such that $r \Phi = \Phi \nu(r)$ for all $r \in \R$,
and it is straightforward to check that this map is actually an algebra automorphism, commonly called \emph{the Nakayama automorphism}.
To phrase it another way, $\nu$ satisfies
\begin{eqnarray}\label{nakprop}
\Phi(qr) = \Phi(\nu(r) q)
\end{eqnarray}
for all $q,r \in \R$. It is clear that $\nu$ is only uniquely determined up to inner automorphism and so determines a class in $\Out(\R) = \Aut(\R) / \Inn(\R)$.
It is a useful invariant to calculate since the class of $\nu$ is trivial if and only if $(q,r) \longmapsto \Phi(qr)$ is a symmetric form $\R \times \R \rightarrow \S$
which, in turn, has rather stark consequences for the representation theory of $\R$; see \cite[Ch. IX]{CR66} for example.

\subsection{The Rees algebra of a filtered algebra}
The Rees algebra $\Rees(\R)$ of a $\Z$-filtered algebra $\R$ is a well known tool from algebraic geometry \cite[6.5]{Ei95}. Since we were unable
to find sources in the literature which work at our level of generality, we shall present some of the details here.

We continue to assume that $(G, \leq)$ is a totally ordered, finitely generated abelian group. We will consider the group algebra of $G$
and so we use multiplicative notation in $G$, writting $1_G$ for the identity element.  Let $\k G$ be the group algebra of $G$ and,
more generally, write $\k P \subseteq \k G$ for any monoid $P \subseteq G$. Let $P := \{g \in G \mid g \geq 1_G\}$ denote the \emph{positive cone of $G$},
and let $\R$ be a non-negatively filtered algebra $\R = \bigcup_{g \geq 1_G} \R_g$.
The Rees algebra of $\R$ is
$$\Rees(\R):= \bigoplus_{g\in G} \R_g \otimes g \subseteq \R \otimes_\k \k P.$$
Choose generators $x_1,...,x_n$ for $G$. After replacing some of these by their inverses we may assume $x_i \geq 1_G$ for all $i$ and so $P$
is the monoid generated by $x_1,...,x_n$. It is easy to see that the group $G$ is torsion free hence $G \cong \Z^n$ for some $n \geq 0$ by the classification
of finitely generated abelian groups, and it follows that $\Spec \k P$ is an $n$-dimensional affine space over $\k$.

The following fact is well known.
\begin{lem}
$\Rees(\R)$ is a free module over $\k P$.
\end{lem}
\begin{proof}
Choose a $\k$-basis $\{b_{g, j} \mid g\in G, j=1,...,\dim \R_g\}$ such that $b_{g, j} \in \R_g$ for all $j$, and observe that
$\{b_{g,j} \otimes g \mid g\in G, j= 1,...,\dim \R_g\}$ is a $\k P$-basis for $\Rees(\R)$.
\end{proof}

\subsection{Factors of the Rees algebra}\label{factRees}

We now consider reductions of $\Rees(\R)$ by maximal ideals $\m \in \Spec \k P$. 
Consider the following two canonical ideals 
\begin{eqnarray*}
& & \m_0 := (g\in P \mid g \gneq 1_G) \in \Spec \k P; \\
& & \m_1 := (g - h \mid g, h \in P,  g\geq h) \in \Spec \k P.
\end{eqnarray*}
It is not hard to see that $\m_1 = (x_1 - 1,...,x_n - 1)$.
\begin{lem}\label{reducingRees}
Continue to assume that the filtration of $\R$ is non-negative. The following hold:
\begin{itemize}
\item[(0)]{$\Rees(\R) / \m_0 \Rees(\R)  \cong \gr \R$;}
\item[(1)]{$\Rees(\R) / \m_1 \Rees(\R) \cong \R$.}
\end{itemize}
\end{lem}
\begin{proof}
First of all observe that $$\m_0 \Rees(\R) = \bigoplus_{g\in G} (\sum_{h > 1_G} \R_{gh^{-1}}) \otimes g = \bigoplus_{g\in G} (\sum_{h < g} \R_{h} ) \otimes g.$$
This immediately leads to (0).

Now observe that there is a homomorphism 
\begin{eqnarray*}
& & \Rees(\R) \longrightarrow \R;\\
& & \sum_{g\in G} r_g \otimes g \longmapsto \sum_{g\in G} r_g.
\end{eqnarray*}
The kernel is generated as a right $\k P$-module by elements $r \otimes (x_i - 1)$ with $i = 1,...,n$ and $r \in \R$. This proves (1).
\end{proof}
\begin{rem}
To complete the picture is is worth mentioning that the quotients $\Rees(\R) / \m\Rees(\R)$ for the remaining maximal ideals $\m \in \Max \k P$ are isomorphic to
the associated graded algebras with respect to filtrations induced by subgroups of $G$.
\end{rem}

\section{Proof of the transfer results}\label{proofof}
Throughout this section $(G, \leq)$ is a totally ordered, finitely generated abelian group and $\R$ is a non-negatively filtered $\k$-algebra with subalgebra
$\S \subseteq \R$. We write $S := \gr \S \subseteq R := \gr \R$. 

\subsection{Passing the Frobenius property through a quotient}
Before we begin we need one easy transfer result.
\begin{lem}\label{quotientslemma}
Suppose that $C \subseteq \S$ is a central subalgebra, $I \unlhd C$ is an ideal (resp. homogeneous ideal) and $\S \subseteq \R$ is a free-filtered (resp. free-graded) Frobenius extension.
Then $\S/I\S \subseteq \R/I\R$ is a free-filtered (resp. free-graded) Frobenius extension of the same rank and degree.
\end{lem}
\begin{proof}
It is easy to see that $\R / I \R$ is free over $\S / I \S$ and a basis is given by the image under $\R \twoheadrightarrow \R/I\R$ of a basis for $\R$ over $\S$.
It quickly follows that $\R/ I\R$ is a free-filtered $\S/I\S$-module.

For $C$-modules $\M$ we may write $\M_I := M/ IM = M \otimes_C C/I$. We claim that $\Hom_\S(\R, \S)_I \cong \Hom_{\S_I}(\R_I, \S_I)$ in a natural way.
First of all we have the map $\Hom_\S(\R, \S)_I \overset{\sim}{\longrightarrow} \Hom_\S(\R, \S_I)$ defined by sending $\phi \otimes c \mapsto (r \mapsto f(r) \otimes c)$,
which is clearly an isomorphism. Next we observe that, since elements of $\Hom_\S(\R,\S)$ are $C$-equivariant we have $\Hom_\S(\R, \S_I) = \Hom_\S(\R_I, \S_I)$.
Finally, this equals $\Hom_{\S_I}(\R_I, \S_I)$ since the $\S$-action on $\R_I$ and $\S_I$ factors through $\S_I$.

Now let $\Phi : \R \rightarrow \S$ be a Frobenius form and $\Psi : \R \rightarrow \Hom_\S(\R, \S)$ the corresponding isomorphism $r \mapsto (q \mapsto \Phi(rq))$
(see Lemma~\ref{characterisation} and the remarks that follow). We define a form $\Phi_I : \R_I \rightarrow \S_I$ by setting $\Phi_I(r + I\R) = \Phi(r) + IS$.
It is well defined since $\Psi$ is $C$-equivariant and $\Psi_I : \R_I \rightarrow \Hom_{\S_I}(\R_I, \S_I)$ is the corresponding map of $\R_I$-$\S_I$-bimodules.
It is readily seen that the following diagram is commutative:
\begin{center}
\begin{tikzpicture}[node distance=1.8cm, auto]
 \node (A) {$\R$};
 \node (E) [right of=A] { };
 \node (B) [right of=E] {$\Hom_\S(\R,\S)$};
 \node (G) [below of=A] { $ $};
 \node (C) [below of= G] {$\R_I$};
 \node (F) [right of=C] { };
 \node (D) [right of=F]{$\Hom_{\S_I}(\R_I, \S_I)$};
 \node (H) [below of=B] {$\Hom_\S(\R,\S)_I$};
  \draw[->] (A) to node {$\Psi$} (B);
 \draw[->>] (A) to node {$ $} (C);
 \draw[->] (C) to node {$\Psi_I$} (D);
 \draw[->>] (B) to node {$ $} (H);
 \draw[->] (H) to node {$\sim$} (D);
\end{tikzpicture}
\end{center}
Since $\Psi$ is surjective it follows that $\Psi_I \circ (-)_I : \R \rightarrow \Hom_{\S_I}(\R_I, \S_I)$ is surjective, and so too is $\Psi_I$.

Both $\R$ and $\Hom_\S(\R,\S)$ are free-filtered and we let $D_1, D_2$ denote their multisets of degrees of free-filtered generators.
By Lemma~\ref{homfilt} and Lemma~\ref{technical} we know that $D_1 d = D_2^{-1}$. By the observations of the first paragraph of the
current proof the $\S_I$-modules $\R_I$ and $\Hom_{\S_I}(\R_I, \S_I)$ are free-filtered and the degrees of the free-filtered generators
are also $D_1, D_2$ respectively. Now we may apply Lemma~\ref{technical} once more to see that $\Psi_I$ is injective. This completes the proof.
\end{proof}

\subsection{Lifting the Frobenius property through a filtration}\label{liftingsection}

Suppose that $S\subseteq R$ is a free-graded Frobenius extension of $\k$-algebras.
\begin{prop}\label{liftingupprop}
$\S \subseteq \R$ is a free-filtered Frobenius extension of the same rank and the same degree as $S \subseteq R$.
\end{prop}
\begin{proof}
Since $S \subseteq R$ is a free-graded Frobenius extension there exists a map of left $S$-modules $\oP : R \rightarrow S$ which satisfies
\begin{enumerate}
\item[(F1$^\prime$)] $\Ker \oP$ contains no non-trivial homogeneous left or right ideals;
\item[(F2)] $r \mapsto (q \mapsto \Phi(rq))$ is surjective $R \rightarrow \Hom_S(R,S)$;
\item[(GF1)] $\oP$ is homogeneous;
\item[(GF2)] $R$ is free-graded as an $S$-module.
\end{enumerate}

Property (GF2) and Lemma~\ref{fgvsff} together imply that $\R$ is free-filtered over $\S$ of the same rank, and that the multiset of degrees
of the free-filtered generators coincides with that of $R$ over $S$.
By Lemma~\ref{identifygraded} we may find some $\Phi \in \Hom_\S(\R,\S)$ with $\oP = \Phi + \sum_{g < \deg(\Phi)}\Hom_\S(\R,\S)_g$.
We claim that $\Phi$ satisfies properties (F1) and (F2) characterising free Frobenius extensions (see Lemma~\ref{characterisation}).

If $\K \subseteq \R$ is any subspace we identify the associated graded $K$ with a subspace of $R$ in the obvious way.
We claim that $\gr \Ker \Phi \subseteq \Ker \oP \subseteq R$.
The space $\gr \Ker \Phi$ is spanned by its homogeneous components. Let $\orr$ be one of these homogeneous elements and let $r \in \Ker(\Phi)$
be such that $\orr = r + \sum_{g<\deg(r)} \R_g$. Then $$\oP(\orr) = \Phi(r) + \sum_{g < \deg(\Phi) + \deg(r)} \S_g = 0.$$
For a general element of $\gr \Ker(\Phi)$ we apply the above reasoning to each of the homogeneous summands, which confirms the claim.
Now let us suppose that $I \unlhd \R$ is an ideal contained within $\Ker (\Phi)$. Consider the graded ideal $\gr I \unlhd R$.
By the previous observations we have $\gr I \subseteq \gr \Ker(\Phi) \subseteq \Ker(\oP)$. By assumption $\Ker(\oP)$ does not contain any non-trivial graded ideals, which forces
$\gr I = 0$ and $I = 0$, which verifies (F1).

We finish the proof with an argument similar to the last paragraph of the previous lemma. Both $R$ and $\Hom_S(R,S)$ are free-filtered
and we write $D_1, D_2$ for their multisets of degrees of free-graded generators. By Lemma~\ref{homfilt} and Lemma~\ref{technical}
we have $D_2 = D_1^{-1}$ and $D_1 d = D_1^{-1}$. By the proof of Lemma~\ref{fgvsff}, along with Lemma~\ref{homfilt} the degrees of
the free-filtered generators of $\R$ and $\Hom_\S(\R,\S)$ over $\S$ are $D_1$ and $D_2$ respectively. Now we may apply
Lemma~\ref{technical} to see that $\Psi_I$ is surjective, which completes the proof.
\end{proof}

\subsection{Passing the Frobenius property up to the Rees algebra}
Suppose that $\S \subseteq \R$ is a free-filtered extension of $\k$-algebras.
\begin{prop}\label{passingupprop} 
$\Rees(\S) \subseteq \Rees(\R)$ is a free-graded Frobenius extension of the same rank and the same degree as $\S \subseteq \R$.
\end{prop}
\begin{proof}
Let $\M$ be a rank one free-filtered $\S$-module, ie. $\M \cong \S$ as $\S$-modules but the filtration has been shifted by some element $d \in G$.
Then we have $\Rees(\M) := \bigoplus_{g\in G} \M_g \otimes g \cong \Rees(\S)(1\otimes d^{-1})$ as $\Rees(\S)$-modules, hence $\Rees(\M)$ is
free-graded over $\Rees(\S)$. Now $\R$ is a finite direct sum of rank one free-filtered modules, hence $\Rees(\R)$ is free-graded over $\Rees(\S)$,
confirming axiom (GF2).

By Lemma~\ref{characterisation} there exists a form $\Phi : R \rightarrow S$ which contains no ideals in its kernel, such that $r \mapsto (q\mapsto \Phi(rq))$
is surjective $\R \rightarrow \Hom_\S(\R,\S)$.
We define
\begin{eqnarray*}
\Rees(\Phi) &:& \Rees(\R) \longrightarrow \Rees(\S);\\
& & \sum_{g\in G} r_g \otimes g \longmapsto \sum_{g\in G} \Phi(r_g) \otimes (g d)
\end{eqnarray*}
For brevity we write $\widetilde \Phi := \Rees(\Phi)$. We clearly have $\widetilde \Phi(\Rees(\R)) \subseteq \S \otimes_\k \k P$,
and since $\Phi(\R_g) \subseteq \S_{gd}$ the image lies in $\Rees(\S)$ as claimed. Furthermore $\tPhi$ is evidently homogeneous
of degree $d$, and so it remains to show that $\tPhi : \Rees(\R) \rightarrow \Rees(\S)$ satisfies properties (F1) and (F2) of Lemma~\ref{characterisation}.

Write $\tPsi : \Rees(\R) \rightarrow \Hom_{\Rees(\S)}(\Rees(\R), \Rees(\S))$ for the map $r \mapsto (q \mapsto \tPhi(rq))$.
In order to show that it is injective it will suffice to check that it is so on the graded components of $\Rees(\R)$.
Suppose that $a \in \R_g$ and $\tPsi(a \otimes g) = 0$. Then for all $q \in \Rees(\R)$ we have $\tPhi((a\otimes g) q) = 0$
In particular we may suppose $b \in \R_h$ and we have $\tPhi((a\otimes g)(b\otimes h)) = \Phi(ab) \otimes gh = 0$
which implies $\Phi(ab) = 0$ for all $b \in \R$. Since $\Phi$ contains no non-zero right ideals we have $a \R = 0$ and so $a = 0$.
This confirms that axiom (F1) of Lemma~\ref{characterisation} holds for $\tPhi : \Rees(\R) \rightarrow \Rees(\S)$.

To check axiom (F2) we use the same argument as per the last paragraph of either of the two previous proofs.

\end{proof}

\subsection{The Transfer Theorem}
Suppose that $\S \subseteq \R$ is a finite, filtered extension so that $S := \gr \S \subseteq R := \gr \R$ is a finite, graded extension.

\begin{thm}\label{transfertheorem}
$\S \subseteq \R$ is a free-filtered extension if and only if $S \subseteq R$ is a free-graded Frobenius extension.
In this case the degrees and ranks of the two extensions coincide.
\end{thm}
\begin{proof}
The `if' part follows from Proposition~\ref{liftingupprop}. Furthermore, Proposition~\ref{passingupprop} tells
us that $\Rees(\S) \subseteq \Rees(\R)$ is a free-graded Frobenius extension of the same rank and degree as $\S \subseteq \R$.
Note that $1\otimes \m_0$ is homogeneous ideal of $\k P$ generated by the generators of $P$ (excluding the identity), therefore $1\otimes \m_0$
generates a homogeneous ideal of $\Rees(\R)$ and of $\Rees(\S)$.
Applying Lemma~\ref{reducingRees} we see that $\Rees(\S)/ \m_0 \Rees(\S) \cong S$ and $\Rees(\R)/\m_0 \Rees(\R) \cong R$
whilst Lemma~\ref{quotientslemma} tells us that $S \subseteq R$ is a free-graded Frobenius extension of the same rank and degree as $\S \subseteq\R$,
which completes the proof.
\end{proof}

\begin{cor}\label{simplifydefn}
$S \subseteq R$ is a free-graded Frobenius extension if and only if it is a Frobenius extension and $R$ is a free-graded $S$-module.
\end{cor}
\begin{proof}
The `only if' part if obvious so suppose that $R$ a is free-graded $S$-module with (not necessarily homogeneous) Frobenius form $\Phi: R \rightarrow S$. Viewing $S\subseteq R$ as a filtered extension we may apply the previous result to deduce that $\gr S \subseteq \gr R$ is a free-graded Frobenius extension. By assumption $\gr S \cong S$ and $\gr R \cong R$.
\end{proof}

\section{Applications: examples and counterexamples}\label{appsec}

\subsection{Quantum Schubert varieties}
Schubert varieties are the closures of the Schubert cells, which provide an affine paving of the flag variety of a reductive algebraic group
over $\C$ and they arise in a vast array of contexts in geometric representation theory; see \cite[Ch. 6]{CG97} for example. As a natural
step in the program of quantising classical geometric objects, their quantum analogues
have been defined and extensively studied (see \cite{DCP93} for example). The quantum coordinate rings on matrices occur as special examples.

Let $\Phi$ be a finite, indecomposable, crystallographic root system with associated Weyl group $W$, let $\g$ be the corresponding complex, simple Lie algebra and let $U_q(\g)$
denote the Drinfeld-Jimbo quantised enveloping algebra generated by $E_1,...,E_r, F_1,...,F_r, K_1^{\pm 1},..., K_r^{\pm 1}$ and with relations which may be read in \cite[Ch. 3]{Lu94}, for example.
Lusztig defined an action on $U_q(\g)$ of braid group $B$ associated to the abstract Weyl group of $\g$. This allows one to construct a (non-minimal)
system of generators for $U_q(\g)$ corresponding to the roots of $\g$ which serve as a system of PBW generators for $U_q(\g)$ (see \cite[I.4.6]{BG02}, for example):
\begin{lem}
Write $\Phi^+ = \{\alpha_1,...,\alpha_N\}$. There exist elements
$$\{E_\alpha, F_\alpha \in U_q(\g) \mid \alpha \in \Phi\}$$
such that the ordered monomials $$E_{\alpha_1}^{a_1} \cdots E_{\alpha_N}^{a_N} K_1^{b_1} \cdots K_r^{b_r}F_{\alpha_1}^{c_1}\cdots F_{\alpha_N}^{c_N}$$
(with $a_i \in \Z_{\geq 0}$ and $b_i \in \Z$) form a $\k$-basis for $U_q(\g)$.
\end{lem}

Pick a set of positive roots $\Phi^+ \subseteq \Phi$. For each Weyl group element $w \in W$ we can consider the space $\Phi[w] := \Phi^+ \cap w \Phi$
and the \emph{quantum Schubert variety} denoted $U_q[w]$ which is the subalgebra of $U_q(\g)$ generated by $\{E_{\alpha} \mid \alpha \in \Phi[w]\}$.
If we suppose that $q\in \C^\times$ satisfies $q^\ell = 1$ then it is well known that $E_{\alpha}^\ell$ is central in $U_q(\g)$. We write $Z_0 \subseteq U_q[w]$
for the subalgebra generated by $\{E_\alpha^\ell \mid \alpha \in \Phi[w]\}$. 

As explained in \cite[Section~10]{DCP93} there is a natural $G$-filtration on $U_q(\g)$ where $G = \Z^{2N + 1}$ with $G$ ordered lexicographically. We shall not describe this
filtration in any detail but we observe that each $U_q[w]$ inherits the subspace filtration. The following is a good illustration of the power of our transfer theorem.
\begin{thm}\label{schubert}
When $q^\ell = 1$ the quantum Schubert variety $U_q[w]$ is a free-filtered Frobenius extension of $Z_0$.
\end{thm}
\begin{proof}
According to \cite[Proposition~10.1]{DCP93} the associated graded algebra $\gr U_q[w]$ is a quantum affine space. The subalgebra $Z_0$ inherits the filtration and it is clear that
$\gr Z_0$ is the subalgebra generated by the $\ell$th powers of the generators. Now apply Proposition~\ref{freegradedqas} and Theorem~\ref{transfertheorem}.
\end{proof}

\begin{rem}
By a very similar argument the quantum Borel $U_q(\mathfrak{b}) \subseteq U_q(\g)$ which is generated by $\{E_\alpha \mid \alpha\in \Phi^+\} \cup \{K_1^{\pm},...,K_r^{\pm}\}$
has an associated graded algebra which is a quantum torus. This allows you to recover Theorems~6.5 and 7.2(2) of \cite{BGS06} in a uniform manner. 
\end{rem}

\subsection{Modular finite $W$-algebras}

Finite $W$-algebras over $\C$ arise via a process of quantum Hamiltonian reduction, and they play a key role in the current developments in the
representation theory of complex semisimple Lie algebras (see \cite{Lo10} for example). The modular analogues of these $W$-algebras were first studied by Premet in \cite{Pr02}
(where the central quotients appeared) and more recently in \cite{Pr10} where the infinite dimensional versions were studied by reduction modulo $p$.
Recently Goodwin and the second author have presented a uniform approach to theory of modular finite $W$-algebras \cite{GT17}.
In the current section we show that the modular finite $W$-algebra is a Frobenius extension of the $p$-centre, with trivial Nakayama automorphism.

Let $G$ be a reductive algebraic group over an algebraically closed field of characteristic $p>0$ and set $\g = \Lie(G)$.
Recall that the enveloping algebra $U(\g)$ contains a large central subalgebra $Z_p(\g)$ called the $p$-centre.
We also assume the standard hypotheses, which can be read in \cite[\textsection 2.2]{GT17} for example, and we let $\kappa : \g \times \g \rightarrow \k$ denote
the non-degenerate $G$-invariant bilinear form on $\g$. Pick a nilpotent element $e\in \g$. According to \cite{Ja04} we can choose
an associated cocharacter $\lambda : \k^\times \rightarrow G$ with the property that $\lambda(t) e = t^2 e$ for all $t\in \k^\times$
and $\lambda(\k^\times)$ acts rationally on the centraliser $\g^e$ with positive eigenvalues. Write $\g(i) \subseteq \g$ for the $i$-eigenspace of $\lambda(\k^\times)$.

The bilinear form $\g(-1) \times \g(-1) \rightarrow \k$ given by $(x, y) \mapsto \kappa(e, [x,y])$ is non-degenerate and, according to \cite[\textsection 4.1]{GT17}, we can choose
a Lagrangian subspace $\l \subseteq \g(-1)$ in such a way that the nilpotent Lie algebra $\m := \l \oplus \bigoplus_{i < -1} \g(i)$ is algebraic,
ie. we may suppose $\m = \Lie(M)$ for some closed connected unipotent subgroup $M \subseteq G$. The linear function $x\mapsto \kappa(e, x)$
defines a character on $\m$ and we write $\m_e = \{x - \kappa(e, x) \mid x\in \m\}$. Now the modular finite $W$-algebra is defined to be
the quantum Hamiltonian reduction $U(\g,e) := (U(\g)/U(\g) \m_e)^{\operatorname{Ad}(M)}$. The $p$-centre of the $W$-algebra is defined to be
$Z_p(\g, e) := (Z_p(\g)/ Z_p(\g) \cap U(\g) \m_e)^{\operatorname{Ad}(M)}$.
\begin{thm}
The $W$-algebra $U(\g,e)$ is a free-filtered Frobenius extension of the $p$-centre with trivial Nakayama automorphism.
\end{thm}
\begin{proof}
The PBW theorem for $U(\g,e)$ states that the associated graded algebra $\gr U(\g,e)$ with respect to the Kazhdan filtration is isomorphic to $\k[e + \v]$ where
$\v$ is a homogeneous complement to $T_e \Ad(G)e$ inside $\g$ \cite[Theorem~5.2]{GT17}. Furthermore, it follows from \cite[Lemmas~5.1 \& 8.2]{GT17} that $\gr Z_p(\g, e)$ identifies with
$\k[e + \v]^p$ as a Kazhdan graded subalgebra of $\k[e+\v]$. It follows immediately that $\gr U(\g,e)$ is a free-graded Frobenius extension of $\gr Z_p(\g,e)$
(this is actually a special case of Proposition~\ref{freegradedqas} where all commutation parameters are 1). By Theorem~\ref{transfertheorem}
we see that $U(\g,e)$ is free-filtered Frobenius extension of $Z_p(\g,e)$.

The enveloping algebra $U(\g)$ has central reductions $U_\eta(\g) := U(\g) / \eta U(\g)$ called reduced enveloping algebras parameterised by
elements $\eta\in \Max Z_p(\g) \leftrightarrow (\g^\ast)^{(1)}$. It is easy to see that the Lie algebras of reductive groups are unimodular, that is to
say that they satisfy $\tr(\ad(x)) = 0$ for all $x \in \g$, and so it follows from \cite[Proposition~1.2]{FP88} that $U_\eta(\g)$ is a symmetric algebra.
Now consider the reduced finite $W$-algebras $U_{\eta}(\g, \chi) := U(\g, \chi)/ \eta U(\g, e)$ where $\eta \in \Spec Z_p(\g, e)$. According to
\cite[Lemma~8.2]{GT17} we may identify the maximal spectrum with a subset of the Frobenius twist $(\chi + \m^\perp)^{(1)} \subseteq (\g^\ast)^{(1)}$
and by \cite[Remark~9.4]{GT17} we have an isomorphism $U_\eta(\g) \cong \Mat_{D}(U_\eta(\g, e))$ where $D = p^{\dim \m}$. It is not hard to
see that for $A$ a finite dimensional $\k$-algebra, $A$ is symmetric if and only if $\Mat_D(A)$ is so. To be precise, if $\mathbb{B} : \Mat_D(A) \times \Mat_D(A) \rightarrow \k$
is a non-degenerate symmetric associative bilinear form and $\iota$ is the idempotent corresponding to the $(1,1)$-entry of $\Mat_D(A)$ then the map
$x, y \mapsto \mathbb{B}(\iota x\iota , \iota y\iota )$ is a nondegenerate symmetric associative form on $A$, which we view as a subalgebra of $\Mat_D(A)$.
It transpires that $U_\eta(\g, e)$ is symmetric for all $\eta \in \Spec Z_p(\g, e)$.

Since $Z_p(\g, e) \subseteq U(\g, e)$ is a Frobenius extension we may denote by $\Phi$ the
Frobenius form. In the notation of Lemma~\ref{quotientslemma}, taking $C = Z_p(\g, e)$, we see that $\Phi_\eta$ is a Frobenius form for
$U_\eta(\g, e)$. By the previous paragraph it follows that $\Phi_\eta$ induces a symmetric form on $U_\eta(\g, e)$.

Now fix $x, y \in U(\g, e)$ and view $U(\g, e)$ as a free $Z_p(\g, e)$-module with basis $b_1,...,b_{r}$. If we write $$\Phi(xy) - \Phi(yx) = \sum_{i=1}^r z_i b_i$$
then it follows from the previous paragraph that $z_i \in \eta$ for all maximal ideals $\eta \in \Spec Z_p(\g, e)$. Since $Z_p(\g, e)$ is a reduced, commutative, affine algebra it is a Jacobson ring
and it follows that $\bigcap_{\eta \in \Spec Z_p(\g)} \eta = 0$ and so we conclude that $\Phi(xy) - \Phi(yx) = 0$ for all $x, y\in U(\g, e)$. We have deduced
that $\Phi$ induces a symmetric bilinear form from $U(\g, e)$ to $Z_p(\g, e)$ which completes the proof.
\end{proof}

\subsection{The Quantum Grassmanian $\mathcal{G}r(2,4)$}

In this final section we use the theory we have developed thus far to give an example of a very natural quantum algebra at an $\ell^\th$ root of unity, which is not a free
Frobenius extension of its $\ell$-centre. The proof relies on the following elementary lemma which follows directly from Lemma~\ref{technical} and Lemma~\ref{homfilt}.
\begin{lem}\label{symmetricdegrees}
Suppose that $\S \subseteq \R$ is a free-graded Frobenius extension of degree $d$ and $D$ is the multiset of degrees of the basis elements of $\R$ over $\S$. Then
$$D = D^{-1} d.$$
$\hfill \qed$
\end{lem}

Let $n<m$ be two integers $q$ a nonzero element of the base field $\k$. We recall that the quantum $(n,m)$-Grassmanian $\mathcal{O}_q := \mathcal{O}_q[\Gr(n,m)]$ is the subalgebra of the quantum matrices $\mathcal{O}_q[\Mat(n\times m, \k)]$ generated by the maximal quantum minors. Let $q$ be a primitive $\ell^\th$ root of unity and consider the subalgebra $Z_0 \subseteq \mathcal{O}_q$ which is generated by the $\ell^\th$ powers of the maximal quantum minors.
Since quantum Grassmanians do not appear to have been studied in the root of unity case we include a brief sketch of the fact that the $\ell$th powers are truly central in $\mathcal{O}_q$.
\begin{lem}
$Z_0$ is a central subalgebra of $\mathcal{O}_q$, which we shall call \emph{the $\ell$-centre}.
\end{lem}
\begin{proof}
For $I \subseteq \{1,...,m\}$ we write $[I]$ for the maximal quantum minor in $\mathcal{O}_q[\Mat(n\times m, \k)]$ with columns indexed by $I$ (and rows $\{1, \dots ,n\}$) and by $X_{i,j}$ the canonical generators of $\mathcal{O}_q[\Mat(n\times m, \k)]$. We prove the stronger result, that $[[I]^\ell, X_{i,j}] = 0$. If $j \in I$ then \cite[Lemma~5.2(a)]{GL02} implies that
$[[I], X_{i,j}] = 0$. If $j \notin I$ then \cite[Lemma~5.2(b)]{GL02} imply that $[I] X_{i,j} = qX_{i,j} [I] + (q - q^{-1}) Z$ where $Z := \sum_{k > j, k \in I} q^{|I \cap [j,k]|} X_{i,k}[I \setminus \{k\} \cup \{j\}]$.
Since $Z$ and $X_{i,j}$ quantum commute (by  \cite[Lemmas~5.2(a) and 5.3(b)]{GL02}) an easy induction gives $[I]^{t} X_{i,j} = q^t X_{i,j} [I]^t + (q^t - q^{-t}) Z$ which leads to $[[I]^\ell, X_{i,j}] = 0$
upon setting $t = \ell$.

\end{proof}

The relations between the generators of $\mathcal{O}_q$ are called the \emph{quantum Pl\"{u}cker relations} and are quite complicated in general. Rigal and Zadunaisky have shown
that $\mathcal{O}_q$ is a quantum algebra with straightening law \cite{RZ12}. One of the many nice consequences of this fact is the existence of a filtration such that the associated graded $\gr \mathcal{O}_q$
has a simple presentation. Since we are in survey mode we do not want to describe this filtration in detail, however we shall describe the associated graded
algebra when $(n,m) = (2,4)$.
\begin{lem} \cite[Theorems~4.9 \& 5.1.6]{RZ12}
There is a filtration of $\mathcal{O}_q$ such that $\gr \mathcal{O}_q$ has generators $\{x_1 = [12], x_2 = [13], x_3 = [23], x_4 = [14], x_5 = [24], x_6 = [34]\}$
with relations
\begin{eqnarray}
& &\label{comrel} x_i x_j = q^{s(i,j)} x_j x_i \text{ for all } 1\leq i, j \leq 6 \\ 
& &\label{strel} x_3 x_4 = q^{t(i,j)} x_2 x_4.
\end{eqnarray}
As a consequence, $\mathcal{O}_q$ has a basis given by the standard monomials
\begin{eqnarray}\label{standardbasis}
\{x_1^{k_1}x_2^{k_2} x_i^{k_i} x_5^{k_5} x_6^{k_6} \mid i= 3 \text{ or } 4 \text{ and } k_1,...,k_6\geq 0\}
\end{eqnarray}
$\hfill \qed$
\end{lem}

The following lemma is a short combinatorial exercise. The proof is omitted for the sake of brevity.
\begin{lem}\label{gradedbase}
$\gr \mathcal{O}_q$ is a free module over $\gr Z_0$ with basis
\begin{eqnarray*}
\left\{x_1^{k_1}x_2^{k_2} x_i^{k_i} x_5^{k_5} x_6^{k_6} \mid \begin{array}{c} i= 3,4 ; 0\leq k_1,...,k_6 < \ell \\ \text{ either } k_2 + k_i < \ell \text{ or } k_i + k_5 < \ell\end{array}  \right\}
\end{eqnarray*}
\end{lem}

\begin{thm}
$\mathcal{O}_q = \mathcal{O}_q[\Gr(2,4)]$ is not a Frobenius extension of its $\ell$-centre $Z_0$.
\end{thm}
\begin{proof}
We suppose for a contradiction that $\mathcal{O}_q$ is a Frobenius extension of $Z_0$. The filtration defined in \cite{RZ12} assigns a filtered degree
to each generator $x_1,...,x_6$ in such a way that $\deg(x_3) + \deg(x_4) = \deg(x_2) + \deg(x_5)$. It is easy to see from Lemma~\ref{gradedbase} that $\gr \mathcal{O}_q$ is a free-graded
algebra over $\gr Z_0$. It follows from Lemma~\ref{fgvsff} that $Z_0 \subseteq \mathcal{O}_q$ is a free-filtered extension and so by Theorem~\ref{transfertheorem} we see that
$\gr Z_0 \subseteq \gr \mathcal{O}_q$ is a free-graded Frobenius extension. In particular, it is a Frobenius extension. Write $\oOq := \gr \mathcal{O}_q$ and $\oZo := \gr Z_0$.
Now it follows from Corollary~\ref{simplifydefn} that $\oZo \subseteq \oOq$ is actually a free-graded Frobenius extension with respect to any grading such that
\begin{itemize}
\item[(i)]{$\oZo$ is a graded subalgebra of $\oOq$; and}
\item[(ii)]{$\oOq$ is a free-graded module over $\oZo$.}
\end{itemize}
We consider the grading on $\oOq$ defined by setting
\begin{eqnarray*}
& & \deg(x_1) = \deg(x_3) = \deg(x_5) = \deg(x_6) = 2;\\
& & \deg(x_2) = \deg(x_4) = 1.
\end{eqnarray*}
Since $\deg(x_3) + \deg(x_4) = \deg(x_2) + \deg(x_5) = 3$ this really does define a grading on $\oOq$ and it is easy to see that $\oOq$ is a free-graded module
over $\oZo$ with basis described in Lemma~\ref{gradedbase}. Using Corollary~\ref{simplifydefn} again we see that $\oZo \subseteq \oOq$ is a free-graded
Frobenius extension with respect to this new grading. By inspection the largest degree of any basis element is $\deg(x_1^{\ell-1} x_3^{\ell-1} x_5^{\ell-1} x_6^{\ell-1}) = 8 (\ell-1)$.
There are precisely 2 basis elements of degree 1 in the basis of Lemma~\ref{gradedbase} and there are no elements of degree $8(\ell-1)-1$. This contradicts
Lemma~\ref{symmetricdegrees}.
\end{proof}

\end{document}